\documentclass[11pt,english]{amsart}

\newcommand{\MK}{\text{MK1}}
\newcommand{\lk}{\text{lk}}
\usepackage[T1]{fontenc}
\usepackage{float}

\usepackage[english]{babel}
\usepackage{amssymb,url,xspace}
\usepackage{graphicx}
\usepackage{subfig}
\usepackage{xcolor}
\usepackage{textcmds}
\usepackage[section]{placeins}
\usepackage{hyperref}

\usepackage{amsmath}
\usepackage{amssymb}
\usepackage{amsfonts}
\usepackage{amsthm}

\theoremstyle{plain}
\newtheorem{theo}{Theorem}
\newtheorem{coro}[theo]{Corollary}

\newtheorem{prop}[theo]{Proposition}
\newtheorem{lemm}[theo]{Lemma}

\theoremstyle{definition}
\newtheorem{defi}[theo]{Definition}

\theoremstyle{remark}
\newtheorem{rema}[theo]{Remark}
\newtheorem{exem}{Example}

\numberwithin{equation}{section}
\numberwithin{theo}{section}

\title[]{Negative definite spin filling and branched double covers}
\date {August 14, 2023}

\author{Soheil Azarpendar}
\address{Mathematical Institute, University of Oxford, Andrew Wiles Building, Radcliffe
Observatory Quarter, Woodstock Road, Oxford, OX2 6GG, UK}
\email{azarpendar@maths.ox.ac.uk} 

\begin{document}

\begin{abstract}
We investigate the negative definite spin fillings of branched double covers of alternating knots. We derive some obstructions for the existence of such fillings and find a characterization of special alternating knots based on them. 
\end{abstract}

\maketitle

\section{Introduction}
Given a non-split link $K$, let $\Sigma(S^3,K)$ denote the branched double cover of $S^3$ along $K$. A filling of $\Sigma(S^3,K)$ is a 4-manifold $X$ with $\partial X = \Sigma(S^3,K)$. One method to construct fillings of $\Sigma(S^3,K)$  is to take a spanning surface of $K$, denoted by $F$, and build the branched double cover of $D^4$ over $F^{+}$, where $F^{+}$ is the properly embedded surface that comes from pushing the interior of $F$ inside $D^4$. We use the term spanning filling to distinguish fillings that can be constructed through this method. One of the most important facts about spanning fillings of branched double covers of links is due to Gordon and Litherland \cite{Gordon1978}. They proved that the intersection form of $\Sigma(D^4,F^{+})$ is equal to the Goeritz form of $F$. \\

A standard choice for a spanning surface of $K$, comes from a checkerboard coloring of the regions in a knot diagram. Considering all of the white (resp.~black) regions in $S^2$ and adding twisted bands between them around each crossing will result in a spanning surface of $K$. We refer to this surface as white (resp.~black) Tait surface and denote it by $F_W$ (resp.~$F_B$). For alternating links, $\Sigma(D^4, F_W^{+})$ and $\Sigma(D^4, F_B^{+})$ give us definite fillings of $\Sigma(S^3,K)$ which we will call white (resp.~black) Tait fillings. Greene \cite{AltGreene} proved that this property gives a topological characterization of alternating links. To fix a standard checkerboard colouring of alternating knots we assume that the white Tait surface is negative definite  and gives us a negative definite filling of the branched double cover. Existence and properties of these fillings are widely studied and bounds on their Betti numbers can be derived from Heegaard Floer and Seiberg--Witten theories.\\

We call an alternating knot $K$ \emph{special} if the black Tait surface is orientable. This construction on special alternating links gives us a spin negative definite filling of the branched double cover. In this paper, we discuss how these fillings can detect special alternating links among all alternating links. This is described in the following theorem.
\begin{theo}\label{specialchar}
Let $K$ be a non-split alternating link and $m$ be the number of unmarked white regions in a reduced alternating diagram. If $X$ is a simply-connected negative definite spin filling of $\Sigma(S^3,K)$ then the following inequality holds: 
$$b_2(X) \leq m.$$
Furthermore, equality will only be achieved when $K$ is special alternating.
\end{theo}
Note that in the rest of the paper we work with decorated diagrams, i.e., diagrams with two marked adjacent regions, and $m$ always represents the number of unmarked white regions in a reduced alternating diagram.\\

The existence of simply-connected spin negative definite fillings of the branched double cover of non-special alternating knots is not trivial. Using inequalities from Heegaard Floer and Seiberg--Witten theories, we develop several obstructions to the existence of such fillings. The main ones are in the form of Theorems \ref{cutbound} and \ref{capbound}. First, we need to explain some notations.\\

In this paper, we work with $W$  and its subgraphs. Since we consider reduced diagrams, these graphs won't contain loops but can have multiple edges. We use $V_{G}$ to denote the vertex set of the graph $G$ and $E_G(,)$ to denote the set of edges between two disjoint subsets of $V_G$. Let $\widetilde{W}$ be the reduced white Tait graph; i.e., the white Tait graph $W$ with the vertex associated to the marked region deleted. A subgraph $C$ of $\widetilde{W}$ is called \emph{characteristic} if it satisfies the following equality:
$$\forall \ v \in V_{\widetilde{W}}: e_{W}(v,C) \equiv deg_{W}(v) \mod 2,$$
where $e_{W}(v,C)$ is defined by the formula 
\begin{equation}\label{eq:1}
  e_{W}(v,C) =
    \begin{cases}
      |E_{W}(\{v\}, V_C)|+deg_W(v) & \text{if $v \in C$}\\
      |E_{W}(\{v\}, V_C)| & \text{if $v \notin C$}.\\
    \end{cases}       
\end{equation}
Using $W$, we will build a Kirby diagram for $\Sigma(S^3,K)$ in \ref{Sec2}. The second term in Equation \ref{eq:1} needs to be considered to account for framing of the components. \\ 

Let $\mathcal{C}_{\widetilde{W}}$ denote the set of characteristic subgraphs of $\widetilde{W}$. We will see that these subgraphs classify spin structures on $\Sigma(S^3,K)$ (Theorem \ref{charspin}). For a non-special alternating link, $W$ contains vertices with odd degree and as a result, a characteristic subgraph can't be empty. 
\begin{theo}\label{cutbound}
Let K be a non-special alternating link with odd determinant. If 
$$min_{C \in \mathcal{C}_{\widetilde{W}}} \ |E_{W}(V_C,V_W \setminus V_C)| \geq |V_W|-1,$$
then $\Sigma(S^3,K)$ doesn't have a simply connected negative definite spin filling.
\end{theo} 
\begin{theo}\label{capbound}
Let K be a non-special alternating link, and $\mathfrak{t} \in Spin^{\mathbb{C}}(\Sigma(S^3,K))$. Let $C$ be the characteristic subgraph associated to $\mathfrak{t}$. If
$$|E_{W}(V_C,V_W \setminus V_C)| \geq 9(|V_W|-1),$$
then $(\Sigma(S^3,K),\mathfrak{t})$ doesn't have a simply connected negative definite spin filling.\\
In particular, if
$$min_{C \in \mathcal{C}_{\widetilde{W}}} \ |E_{W}(V_C,V_W \setminus V_C)| \geq 9(|V_W|-1),$$
then $\Sigma(S^3,K)$ doesn't have a simply connected negative definite spin filling.
\end{theo}
While investigating Theorem \ref{capbound}, we find certain obstructions for 4-manifolds to have chainmail Kirby diagrams (See Section \ref{Sec4} for details). This leads to Corollary \ref{characteristicS2} which states that any closed 4-manifold with a chainmail Kirby diagram has a characteristic embedded sphere.\\

Theorem \ref{cutbound} and \ref{capbound}  turn out to be generalized versions of a known obstruction of negative definite spin fillings given by the Neumann--Siebenmann invariant. See Remark \ref{Ue}. \\

We should point out that both of these theorems result in a \qq{twisting phenomenon} in the sense that if you enlarge the twist regions of the alternating link enough, you'll end up with knots whose branched double cover doesn't have a negative definite spin filling. Enlarging twist regions is equivalent to turning and edge $e=\{v_i,v_j\}$ in $W$ to a family of parallel edges between $\{v_i,v_j\}$ in a small tubular neighborhood of $e$, or in terms of the dual graph, turning a vertex to a path of vertices (See Figure \ref{twisting}).\\

Both of the inequalities can be seen as weaker versions of Elkies's condition \cite{Elkies1999ACO} i.e. non-existence of short characteristic vectors.\\
\begin{figure}[h]
\centering
\includegraphics[scale=0.32]{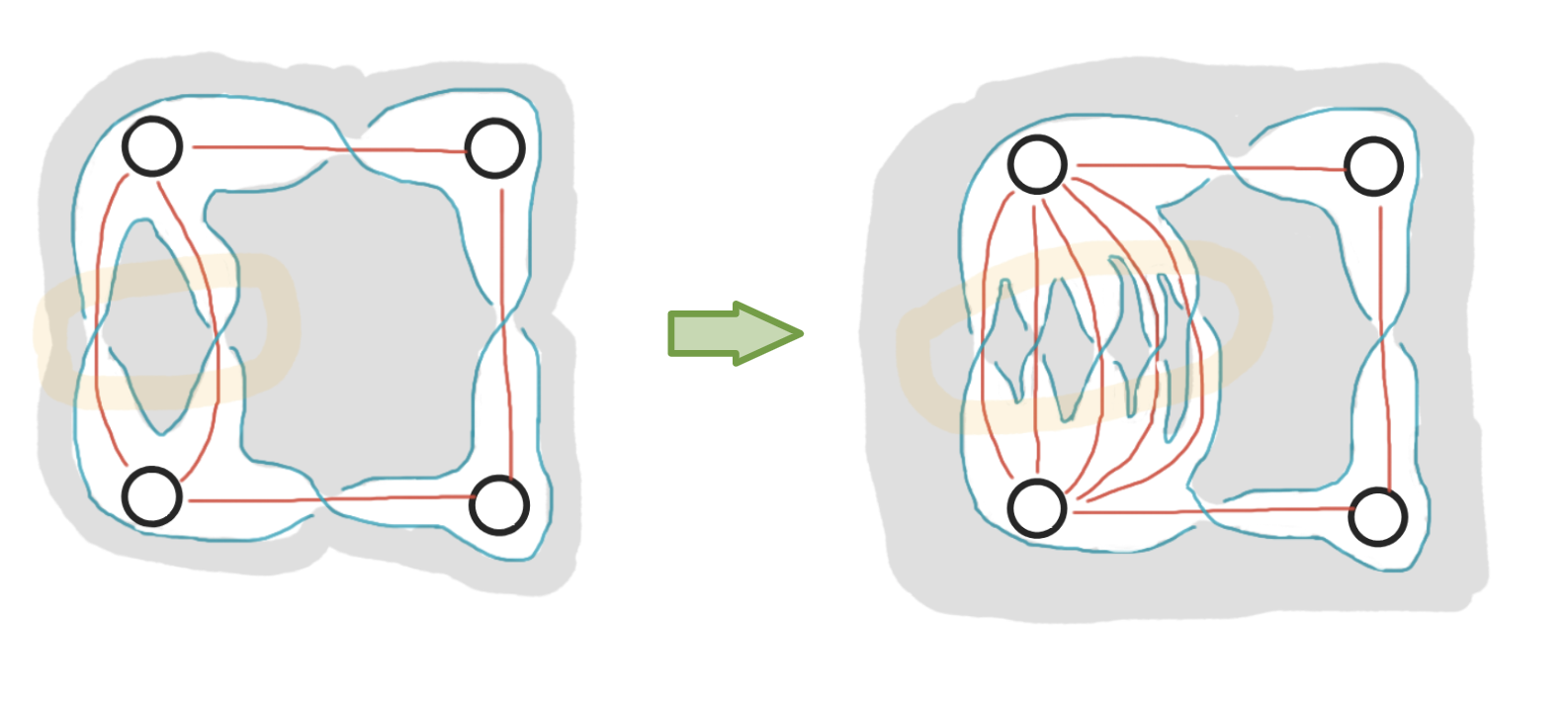}
\caption{Enlarging a twist region}\label{twisting}
\end{figure}

In general, there exist non-special alternating knots with branched \mbox{double} covers which admit simply-connected negative definite spin fillings, but it turns out that by further restricting the search to the class of plumbed 4-manifolds, one can prove a non-existence result explained in the following theorem. First, we need to explain some notations.\\

Theorem \ref{plumbed} is about algebraic links. By algebraic we mean arborescent. Based on the work of Siebenmann \cite{Sieb}, we know that these are the only links where the branched double cover has a plumbed filling. We call an algebraic link \emph{excessive}
if it is constructed from a plumbing tree $T$ with weight function $w$, satisfying the following inequality 
$$ \forall a_i \in V_T : \ w(a_i) \leq min \{-2 , -deg_{T}(a_i)\}.$$
Excessiveness is a technical condition defined by Murasugi \cite{murasugi}. We use it to ensure that the algebraic link is alternating and relate the tree $T$ to the Tait (See Lemma \ref{excessive}).

\begin{theo}\label{plumbed}
    Let K be an excessive algebraic alternating link with odd determinant. Then $\Sigma(S^3,K)$ admits a simply connected spin negative definite plumbed filling if and only if $K$ is special.
\end{theo}

This paper is organized as follows. In Section \ref{Sec2}, we set our basic notations and recall some of the theorems from the literature. This section will include the construction of a Kirby diagram for the white Tait filling, some facts about the Heegaard-Floer homology of $\Sigma(S^3,K)$, and some inequalities about fillings of rational homology spheres and closed spin 4-manifolds. In Section \ref{Sec3}, we discuss proofs of Theorems \ref{specialchar} and \ref{cutbound} which come from a formula for the correction term of the branched double cover. In Section~\ref{Sec4}, we discuss an algorithm of Kaplan which helps us to construct spin fillings and combine it with Furuta's 10/8 theorem to prove Theorem \ref{capbound}. In Section \ref{Sec5}, we will discuss Neumann's plumbing calculus and use it to prove Theorem~\ref{plumbed}. 
\section*{Acknowledgement}
It is a pleasure to thank my advisor, Professor András Juhász, as without his patience and guidance this project wasn't possible. I am also very grateful to Professor Marco Golla for our  helpful discussion and for pointing out Example~\ref{charsphere}.
\section{Background and Notations}\label{Sec2}
Let $K \subset S^3$ be an alternating knot and let $D$ be an alternating diagram of $K$ in the plane. Since the diagram is alternating, one can construct a checkerboard colouring of the diagram such that all the crossings have $\mu=-1$ using the notation of Gordon and Litherland \cite{Gordon1978}. The coloring will look like Figure \ref{crossing} around each crossing.\\
\begin{figure}[h]
\centering
\includegraphics[scale=0.35]{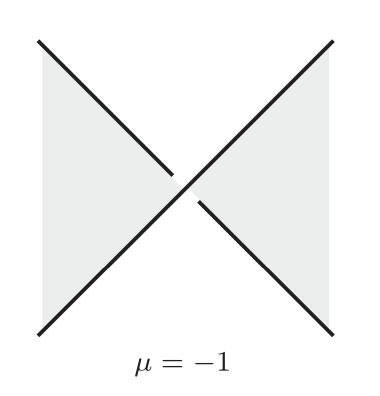}
\caption{Standard coloring of a crossing in an alternating link}\label{crossing}
\end{figure}

In this setting, one can define \emph{white and black Tait graphs and Tait surfaces}. Tait graphs are constructed by considering regions with the same color as vertices and drawing an edge between two regions if and only if they have a common crossing on their boundary. Tait surfaces come from applying median construction to Tait graphs. We use the notations $W$ and $B$ for the graphs and $F_W$ and $F_B$ for the spanning surfaces. In this paper, we assume that diagrams are always decorated, i.e., they have a marked edge or equivalently two marked adjacent regions. We refer to the graphs resulting from deleting the vertices associated to the marked regions as the reduced Tait graphs, and denote them by $\widetilde{B}$ and $\widetilde{W}$. An alternating knot is called \emph{special} if the black Tait surface is orientable, i.e., a Seifert surface. This is equivalent to the black Tait graph being bipartite. Since black and white Tait graphs are dual planar graphs, this definition is also equivalent to the white Tait graph having no vertex with odd degree. Here is an example of a special alternating knot and its Tait graphs:\\
\begin{figure}[!ht]
    \centering
    \subfloat{{\includegraphics[width=3.9cm]{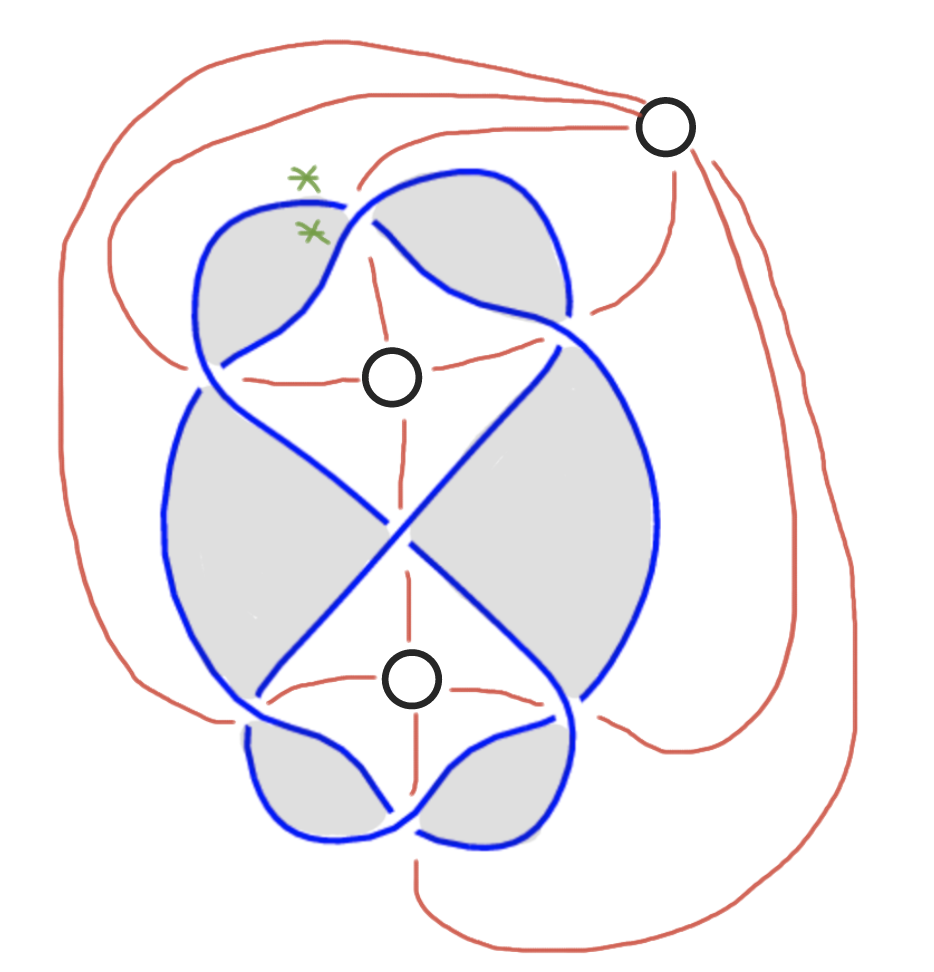} }}
    \qquad
    \subfloat{{\includegraphics[width=2.7cm]{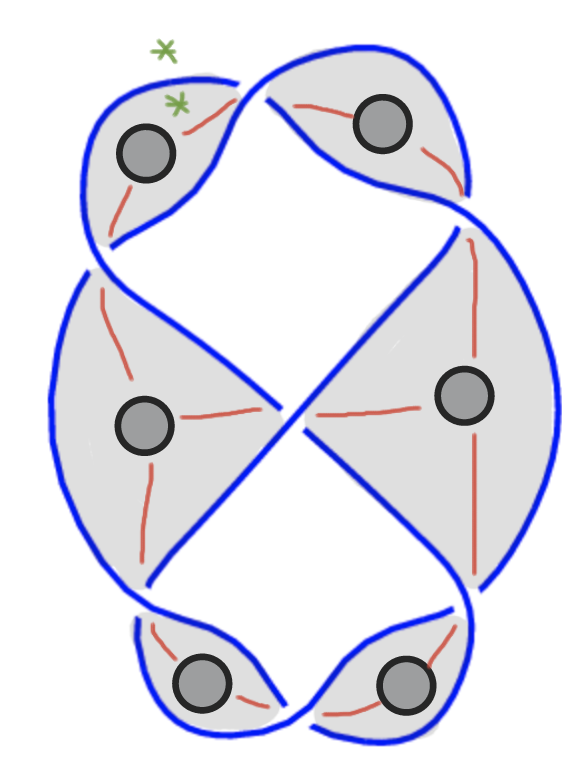} }}
    \label{fig:example}%
\caption{A special alternating knot and its Tait graphs}\label{specialalt}
\end{figure}
\FloatBarrier

As mentioned in the introduction and by works of Gordon and Litherland \cite{Gordon1978}, the branched double cover of $D^4$ over the white Tait surface $F_{W}$, which is denoted by $\Sigma(D^4,F_W)$, is a negative definite filling of $\Sigma(S^3,K)$. The intersection form of $\Sigma(D^4,F_W)$ turns out to be the Goeritz form of $F_W$. There is a combinatorial description of the Goertiz form of the white Tait surface of an alternating knot. Enumerate the vertices of $\widetilde{W}$ by $v_1,\dots, v_n$ and set 
$$g_{ij} := |E_{W}(v_i ,v_j)| \ \text{for} \ i \neq j \ and \ g_{ii} = -deg_{W}(v_i).$$
Then the white Goeritz matrix $G_W:=(g_{ij})$ represents the Goeritz form. Note that this is the definition of the Laplacian matrix of the graph $W$ with the row and column associated with the marked vertex deleted. For the example, in Figure \ref{specialalt},\\
\[
G_W
=
\begin{pmatrix}
    -4 & 1\\
     1 & -4
\end{pmatrix}.
\]
Now we are going to describe a Kirby diagram of $\Sigma(D^4,F_W^{+})$ which also acts as a diagram for the $\Sigma(S^3,K)$. In the rest of the article, we refer to this construction as \emph{Tait surgery diagram}. This diagram originates from \cite{SObranch}.  \\

We consider an unknot component with framing $g_{ii}$ centred around each $v_i \in V_{\widetilde{W}}$ and then add a positive clasp between the unknot components corresponding to $v_i$ and $v_j$ for each edge $e \in E_{\widetilde{W}}$ between $v_i$ and $v_j$. The intersection form of this Kirby diagram is clearly the same as $G_W$. Applying this to the example in Figure \ref{specialalt} gives us Kirby diagram in Figure \ref{Tait-Kirby}. 
\begin{figure}[h]
\centering
\includegraphics[scale=0.4]{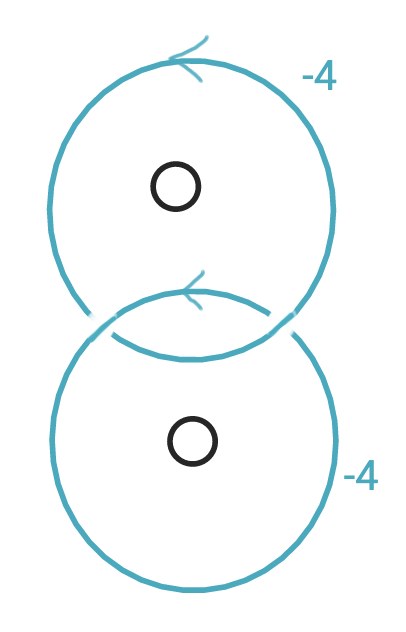}
\caption{Kirby diagram of the branched double cover}\label{Tait-Kirby}
\end{figure}\\
Greene \cite{Greenebrancheddouble} derives a Heegaard triple $\mathcal{H}_1$ subordinate to this surgery diagram and in combination with another Heegaard triple $\mathcal{H}_2$ originating from the Montesinos trick, he gives a combinatorial description of $\widehat{HF}(\Sigma(S^3,K))$. We only need some of Greene's results about alternating links which we will summarize in the following. \\

Given a Kauffman state $x$ for $K$, we induce an orientation on the white graph $W$ in the following way. Given an edge $e \in E_W$, consider the crossing $c$ to which it corresponds, as well as the white region which abuts $c$ and lies on the same side of the over-strand as $x(c)$. We direct $e$ to point towards the vertex corresponding to this white region. At a vertex $v \in V_{\widetilde{W}}$, we compute the signed degree $d_x(v)$ as the number of edges directed into $v$ minus the number of edges directed out of $v$, with respect to this orientation on $W$. Define $v^{W}_{x} := (d_x(v_1), \dots , d_x(v_n))^{T}$. Let $q(v)$ denote the quadratic form $v^TG_{W}^{-1}v$. A \emph{characteristic covector} $v \in \mathbb{Z}^m$ is defined by the condition
that $v_i \equiv (G_W)_{ii} \ mod  \ 2 \  \forall i$.
\begin{theo}\label{Greene1}\cite{Greenebrancheddouble}
    Let K denote a non-split alternating link. Then $\Sigma(S^3,K)$ is an $L$--space. Kauffman states of $K$ are in one-to-one correspondence with $Spin^{\mathbb{C}}$ structures on $\Sigma(S^3,K)$ and contribute a $\mathbb{Z}$ summand to $\widehat{HF}(\Sigma(S^3,K))$. Also, $Spin^{\mathbb{C}}(\Sigma(S^3,K))$ can be identified with $2im(G_W)$--orbits of characteristic covectors for $G_W$. Let $\mathfrak{t}(x)$ denote the $Spin^{\mathbb{C}}$ structure corresponding to $x$. The correction term can be computed by the formula
    $$d(\Sigma(S^3,K),\mathfrak{t}(x))=max_{v \in \mathfrak{t}(x)} \frac{q(v)+m}{4}=\frac{q(v^{W}_{x})+m}{4},$$
    where $m=|V_W|-1$. Furthermore, when $det(K)$ is odd, the first Chern class $c_1$ is a canonical identification of $Spin^{\mathbb{C}}(\Sigma(S^3,K))$ and $coker(G_W)$. This identification also satisfies the following: 
    $$c_1(\mathfrak{t}(x)) = [v^{W}_{x}] \in coker(G_W).$$
    Note that $H^2(\Sigma(S^3,K))$ can be identified with $coker(G_W)$ using the Tait surgery diagram. 
\end{theo}
We are going to use these formulas to obstruct branched double covers from having spin negative definite fillings. To accomplish this, we are going to use known inequalities about fillings of rational homology spheres and closed spin 4-manifolds. We recall some of the theorems that we are going to use later. 
\begin{theo}\label{OSineq}\cite{OSineq}
Let $Y$ be a rational homology three-sphere, and fix a $Spin^{\mathbb{C}}$ structure $\mathfrak{t}$ over $Y$. Then, for each smooth, negative definite filling $X$ of $Y$, and for each $\mathfrak{s} \in Spin^{\mathbb{C}}(X)$ with $\mathfrak{s}|_{Y}=\mathfrak{t}$, we have that 
$$c_1(\mathfrak{s})^2+b_2(X) \leq 4d(Y,\mathfrak{t}).$$
\end{theo}

\begin{theo}\label{Furuta}\cite{Furuta}
If $M$ is a closed spin manifold with indefinite intersection form, then $$b_2(M) \geq \frac{10}{8}|\sigma(M)|+2.$$
\end{theo}
\section{Bounds from correction terms}\label{Sec3}
Now we are ready to prove Theorem \ref{specialchar} and Theorem \ref{cutbound}. 
\begin{proof}[Proof of Theorem \ref{specialchar}]
Assume $X$ is a simply connected, spin, negative definite filling of $Y=\Sigma(S^3,K)$. Recall that $Spin^{\mathbb{C}}$ structures of $X$ correspond to integral lifts of the second Stiefel--Whitney class under the second map in the exact sequence $$H^2(X;\mathbb{Z})\xrightarrow{\times 2} H^2(X;\mathbb{Z})\rightarrow H^2(X;\mathbb{Z}_2)\xrightarrow{\beta} \cdots.$$ The Chern class of a $Spin^{\mathbb{C}}$ structure can be computed as a further lift of the Stiefel--Whitney class to the first group in the sequence. Since $X$ is Spin, the second Stiefel--Whitney class vanishes and, as a result, one can find a trivial lift $\mathfrak{s} \in Spin^{\mathbb{C}}(X)$ with $c_1(\mathfrak{s})=0$. Based on Theorem \ref{Greene1}, there is a Kuaffman state $x$ such that $[v^W_{x}] \in coker(G_W)$ is identified with $\mathfrak{s}|_{Y} \in Spin^{\mathbb{C}}(\Sigma(S^3,K))$. Using Theorem \ref{OSineq}, we can write
$$b_2(X) \leq 4d(Y,\mathfrak{s}|_{Y}) = q(v^{W}_{x})+m \leq m.$$
The last inequality follows from the fact that $q$ is a negative definite form as its defined using inverse of Goeritz matrix.\\

Now we are going to prove that the last inequality is sharp if $K$ is not special. This again comes from Theorem \ref{Greene1}. We only need to show that, for all Kauffman states $x$ on a non-special alternating knot, $q(v^W_x) < 0$. Since $q$ is negative definite, we only need to prove that $v^W_x \neq 0$. This follows from the fact that, for non-special knots, $W$ contains at least two vertices with odd degrees, since it's dual can't be bipartite. As a result, there exist $v_i \in V_{\widetilde{W}}$ such that $deg_W(v_i)$ is odd. On the other hand, $v^W_x$ is a characteristic covector; i.e.,
$$v^W_x=(d_x(v_i))=(d^{+}_{x}(v_i)-d^{-}_{x}(v_i))\equiv (deg_W(v_i)) = -g_{ii} \  mod \ 2.$$

Finally, we need to show that, for special alternating knots, there exists a simply connected, spin and negative definite filling of $\Sigma(S^3,K)$ with $b_2=m$. The 4-manifold $X=\Sigma(D^4,F_W^{+})$ satisfies these conditions. Using the Tait surgery diagram (which is a Kirby diagram of $X$), one can see that $X$ is simply connected with $b_2(X)=m$ and negative definite with intersection form $Q_X=G_W$. Furthermore $Q_X$ is even as $g_{ii} = deg_W(v_i)$ which is even since $K$ is special. Hence, $X$ is spin as well. 
\end{proof}
\begin{proof}[Proof of Theorem \ref{cutbound}]
This proof is similar to the previous one. Based on Theorem \ref{Greene1}, there is a Kuaffman state $x$ such that $[v^W_{x}] \in coker(G_W)$ is identified with $\mathfrak{s}|_{Y} \in Spin^{\mathbb{C}}(\Sigma(S^3,K))$. Using Theorem \ref{OSineq}, we can write
$$b_2(X) \leq 4d(Y,\mathfrak{s}|_{Y}) = q(v^W_{x}) + m.$$
Note that $\mathfrak{s}$ is induced by a $Spin$ structure on $X$, and, as a result, $\mathfrak{s}|_{Y}$ is also the $Spin^{\mathbb{C}}$ structure induced by the unique $Spin$ structure on $\Sigma(S^3,K)$~. The uniqueness of the $Spin$ structure follows from the assumption that $det(K)=|H_1(\Sigma(S^3,K) ; \mathbb{Z})|$ is odd and, as a result, has no 2-torsion and $H^1(\Sigma(S^3,K) ; \mathbb{Z}_2)$ vanishes. Based on this argument and the one made in the first lines of the proof of Theorem \ref{specialchar}, we can deduce that $c_1(\mathfrak{s}|_{Y})=0$.\\

We will show that the inequality in the statement of Theorem \ref{cutbound} results in $d(Y,\mathfrak{s}|_{Y})$ being negative. Based on Theorem \ref{Greene1}, we know that $Spin^{\mathbb{C}}(\Sigma(S^3,K))$ is identified with $coker(G_W)$ through the first Chern class. As a result, $[v^W_{x}]=[0]=[c_1(s|_{Y})] \in coker(G_W)$, which means that there exists $y \in \mathbb{Z}^m$ such that $G_{W}y=v^W_{x}$. Note that we can rewrite $$q(v^W_{x}) = (v^W_{x})^T G_W^{-1} v^W_{x} = y^T G_W^T y= y^T G_W y.$$ We know that $v^W_{x}$ is a characteristic covector and by definition we have
\begin{equation}\label{characteristic}
(v^W_{x})_i \equiv -g_{ii} \ mod \ 2.
\end{equation}
We can rewrite Equation \ref{characteristic} as 
$$(v^W_{x})_i = (G_Wy)_{i} = \Sigma_{j} g_{ij}y_j \equiv deg_W(v_i) = -g_{ii} \ mod \ 2.$$
Let $y'\in\mathbb{Z}^m$ be the vector defined as $y'_i = (y_i \ mod \ 2) \in \{0,1\}$. Let $J$ be the support of $y'$. We use $C$ to denote the subgraph of $W$ induced by $v_j$ for $j \in J$. Then one can see that 
$$deg_W(v_i) \equiv (G_Wy)_{i} \equiv (G_Wy')_i \equiv e_{W}(v_i,C)\ mod \ 2.$$
The last equality follows from the properties of the Laplacian matrix. Indeed, we have that 
$$(G_Wy')_i = \sum_{j \in J} g_{ij} = \sum_{j \in J-\{i\}} g_{ij} + \sum_{i \in J} g_{ii} = \sum_{j \in J} |E(\{v_i\},\{v_j\})| - \sum_{i \in J} deg_W(v_i),$$
which is equal to $e(v_i , C)$ mod 2. As a result, $C$ is a characteristic subgraph. Note that, since we assume the knot to be non-special, $C$ can't be empty.\\

We can also use the interpretation of $G_W$ as a submatrix of the Laplacian of $W$ to reformulate $y^TG_Wy$ as the following sum. Assume that $L_W$ is the full Laplacian of $W$ with the $(m+1)$-th (last) row and column associated to the distinguished vertex and set $y_{m+1}=0$. Then we have
\begin{equation}\label{laplaciansum}
    y^T G_W y = [y^T \ 0] L_W \begin{bmatrix} y \\ 0 \end{bmatrix} = -\Sigma_{\{v_i,v_j\}\in E_W} (y_i-y_j)^2.
\end{equation}
Based on the definition of $C$, we will have $y_i \neq y_j$ for $v_i \in V_C$ and ${v_j \in V_W \setminus V_C}$. Combined with Equation \ref{laplaciansum}, we will have that
$$q(v^W_{x}) =y^T G_W y \leq -|E(V_C,V_W \setminus V_C)|.$$
Combining this with the statement of the Theorem \ref{cutbound}, we have that ${q(v^W_{x})\leq -m}$, which gives us the result that we want.
\end{proof}
\begin{rema}
As we will mention in the next section, the definition of characteristic subgraph is an analogue of the definition of characteristic sublink in a Kirby diagram. Characteristic sublinks in a Kirby diagram of $Y$ are in one-to-one correspondence with $H^1(Y;\mathbb{Z}_2)$ and, as a result, in the setting of Theorem \ref{cutbound}, there is only one characteristic sublink in the white Tait graph.
\end{rema}
\section{Bound from Furuta's 10/8 theorem}\label{Sec4}
It's well known that the third spin cobordism group vanishes. This means that any spin 3-manifold $(Y,\mathfrak{t})$ has a spin filling; i.e., there exists a spin 4-manifold $(W,\mathfrak{s})$ such that $\partial W=Y$ and  $\mathfrak{s}|_{Y}=\mathfrak{t}$. In fact, Kaplan \cite{Kaplan} built an algorithm that turns any Kirby diagram which doesn't contain any 1-handles for $Y$ to a diagram of a spin filling through Kirby calculus. We will recall this algorithm from its fantastic exposition in the book by Gompf and Stipsicz \cite{GompfStip}. 

\begin{theo}[\cite{GompfStip}]\label{charspin}
Assume that we have a Kirby diagram $L$ of a 3-manifold $Y$ which doesn't contain any 1-handles. This also gives us a handlebody filling $X_L$ of $Y$. For any spin structure $\mathfrak{t}$ on $Y$, the obstruction for extending $\mathfrak{t}$ to $X$ is the relative Steifel--Whitney class $w_2(X,\mathfrak{t}) \in H^2(X,Y;\mathbb{Z}_2)$, which gives a bijection between spin structures on $Y$ and characteristic sublinks of $L$. (Each $L' \subset L$ defines an element of $H^2(X,Y;\mathbb{Z}_2)$ with $[H_i] \rightarrow \lk(L_i,L) \ mod \ 2$, where $[H_i] \in H_2(X)$ comes from capping off core of the 2-handle attached along $L_i$.)
\end{theo}

\begin{prop}[\cite{GompfStip}]\label{Kaplan}
Assume that we have a Kirby diagram $L$ of a 3-manifold $Y$ which doesn't contain any 1-handles. Fix any spin structure $\mathfrak{t}$ on $Y$ and assume that $L'$ is the corresponding characteristic sublink. The following steps will lead to a Kirby diagram of a spin filling of $(Y,\mathfrak{t})$:\\

1. Slide one component $K'$ of $L'$ over the others. The characteristic sublink corresponding to $\mathfrak{t}$ in the new Kirby diagram  
will be the sublink consisting of one component $K'$.\\

2. Unknot $K'$ using blow-ups. The blow-up circles can be imagined as connected sum of two small meridian circles $m_{i1},m_{i2}$ along a band $b_i$ forming $D_i$ as in Figure \ref{Kaplan2}.
\begin{figure}[H]
\centering
\includegraphics[scale=0.4]{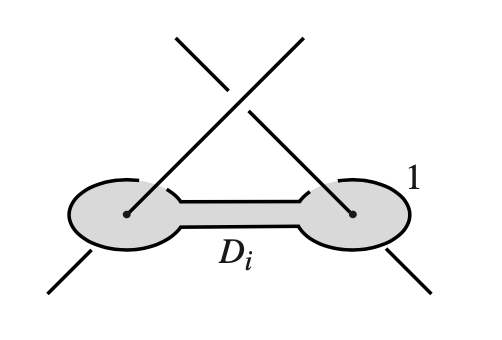}
\caption{Step 2 in Kaplan's algorithm\ \cite{GompfStip}}\label{Kaplan2}
\end{figure}

The characteristic sublink will be the union of $K'$ and all the blown-up circles. \\

3. One can again use blow-ups to change the crossings between the bands and components of $L$, without changing the characteristic sublink, See \mbox{Figure \ref{Kaplan3}}.
\begin{figure}[H]
\centering
\includegraphics[scale=0.3]{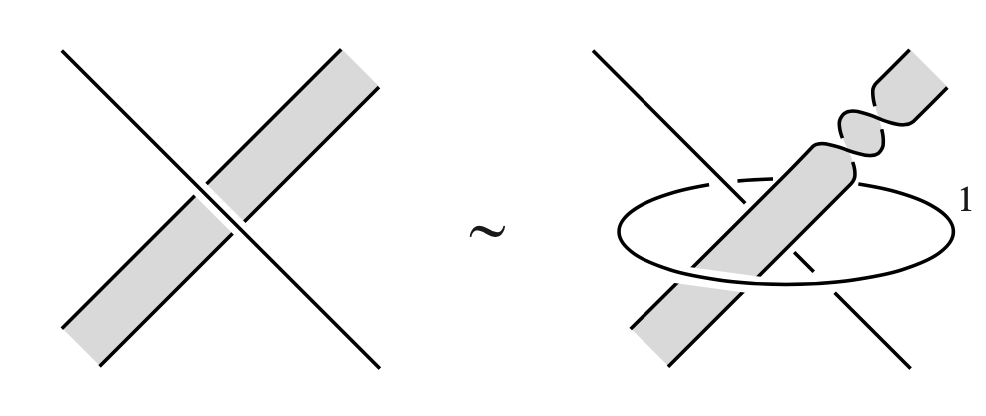}
\caption{Step 3 in Kaplan's algorithm\ \cite{GompfStip}}\label{Kaplan3}
\end{figure}
Use an isotopy to turn $K'$ into a circle in the plane and then use the operation of Figure \ref{Kaplan3} to turn the characteristic sublink into Figure \ref{Kaplan3'}.\\ 

\begin{figure}[H]
\centering
\includegraphics[scale=0.3]{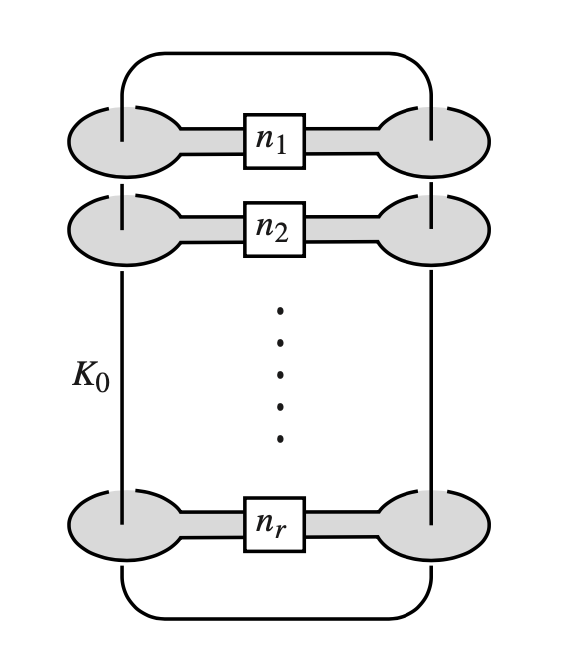}
\caption{Step 3 in Kaplan's algorithm\ \cite{GompfStip}}\label{Kaplan3'}
\end{figure}
4. The operation shown in Figure \ref{Kaplan4} can be done using a blow-up.
\begin{figure}[H]
\centering
\begin{center}
\includegraphics[scale=0.3]{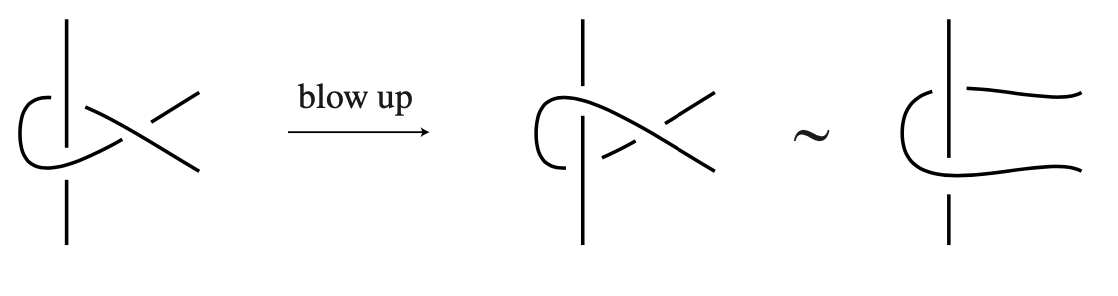}
\end{center}
\caption{Step 4 in Kaplan's algorithm\ \cite{GompfStip}}\label{Kaplan4}
\end{figure}
Use this operation to turn the characteristic sublink into an unlink.\\

5. Consider each component of the characteristic sublink one by one. Blowing up its meridians turns the framing to $\pm 1$. Then, by blowing down the characteristic sublink, one can turn it into the empty link. \\

In the end, you'll have a Kirby diagram with even farmings. This \mbox{4-manifold} with its unique spin structure is a spin filling of $(Y,\mathfrak{t})$.
\end{prop}
We are going to show that, in the setting of Theorem \ref{capbound}, Kaplan's algorithm simplifies and, as a result, one can compute the change in the signature and second betti number and prove the obstruction of Theorem \ref{capbound}. We will use Lemma \ref{Kaplansimp} in the proof.\\

Before we state the lemma we need to introduce some notations. We call a framed link a \emph{chainmail link} if it's constructed in the following way. Let $D$ be a weighted and signed plane graph. Graph $D$ can have loops and multiple edges. Assigned to each $v_i \in V_D$, there is an integer weight $w_i \in \mathbb{Z}$ and assigned to each $e_k \in E_D$, there is a sign $\mu_k \in \{+ , -\}$. The framed link $L_D$ is constructed by considering an unknot component $L_i$ oriented counter-clockwise and with framing $w_i$ centred around each $v_i \in V_D$ and then add a left-handed (resp.~right-handed) clasp between the unknot components corresponding to $v_i$ and $v_j$ for each edge $e_k \in E_D$ between $v_i$ and $v_j$ with $\mu_k = +$ (resp.~$-$). This definition generalizes the construction of Tait surgery diagram explained in Section \ref{Sec2}. For more on these see \cite{Polyak}.\\

Now we are going to explain a modification of the first step of Kaplan's algorithm. This procedure is called $\emph{MK1}$ and is defined in Definition \ref{MK1}. 

\begin{defi}\label{MK1}
Let $L_D$ be the chain mail link based on the connected plane graph $D$ and $T$ be a spanning tree of $D$. Fix an arbitrary orientation (on each edge) and a total order $<$ on $E_T$. Consider the following procedure: \\
1. Take the maximal edge $e_m$ of $T$ in the ordering and assume it is directed from $v_p$ to $v_t$. Slide $L_p$ over $L_t$ with an orientation-preserving band.\\
2. Contract $e$ in $T$. The two vertices at two ends of $e$ will form a new vertex which we denote by $v_p$ (this labeling will be important in repeating step 1).\\
3. Repeat the process until $|V_T|=1$.\\
The knot corresponding to the remaining vertex is denoted by $\MK(L_D,T)$.
\end{defi}
\begin{lemm}\label{Kaplansimp}
Let $L_D$ be the chain mail link based on the connected plane graph $D$.  There exist a spanning rooted tree $T$ with a total ordering and direction on edges, such that $\MK(L_D,T)$ is an unknot
\end{lemm}
\begin{proof}[Proof of Lemma \ref{Kaplansimp}]
We construct $T$ by induction on the number of vertices in $D$. For any two adjacent vertices $v_i,v_j$ in $D$, Let $ER(v_i,v_j)$ be the maximal bounded region in the plane bounded by the edges between $v_i$ and $v_j$ with respect to inclusion. Pick $r,s$ such the $ER(v_r,v_s)$ is minimal; i.e., it doesn't contain $ER(v_i,v_j)$ for any $i,j$. Note that a clear example of such a minimal pair is a pair of adjacent vertices which only have one connecting edge.\\
\begin{figure}[h]
\centering
\begin{center}
\includegraphics[scale=0.3]{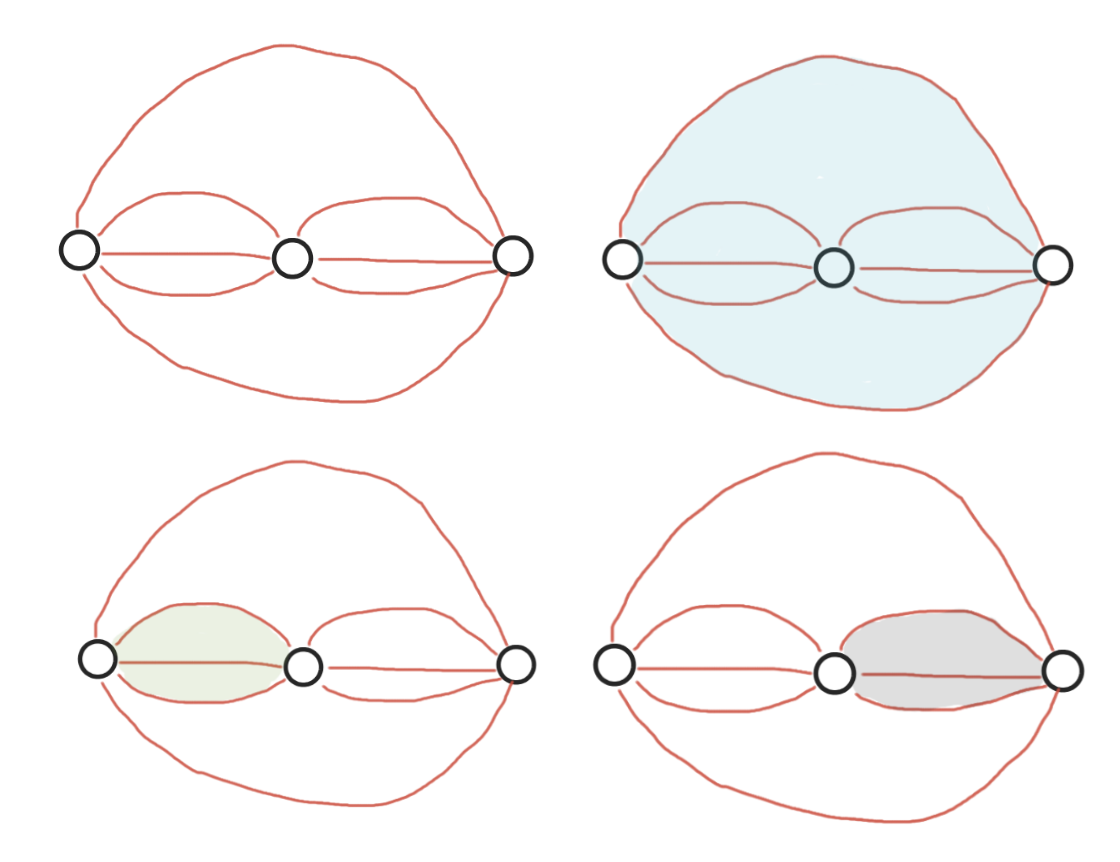}
\end{center}
\caption{Examples of $ER(v_i,v_j)$, the green and black regions are minimal}\label{MK1pic}
\end{figure}
\begin{figure}[h]
\centering
\begin{center}
\includegraphics[scale=0.3]{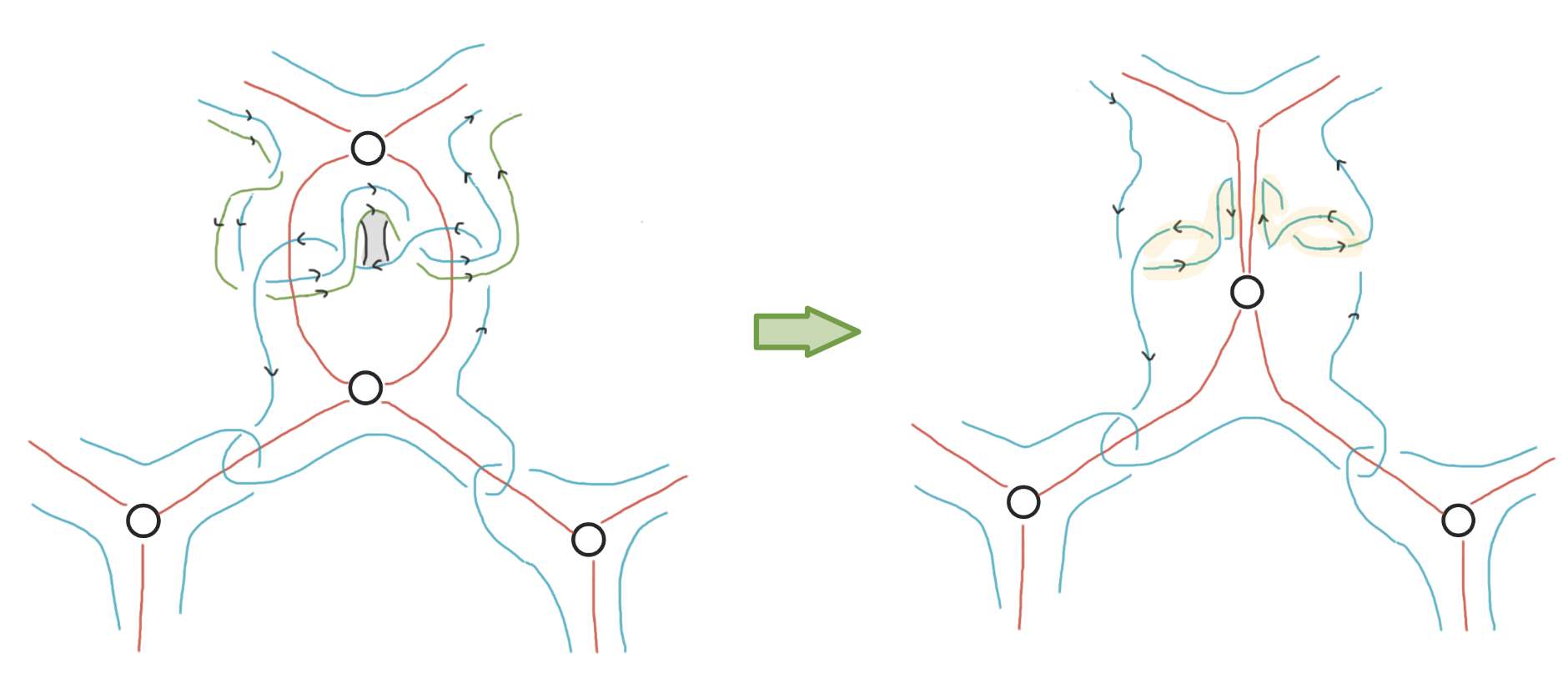}
\end{center}
\caption{Handle slides in \MK}\label{MK1pic}
\end{figure}\\

Picking this minimal pair guarantees that all the other vertices lie in $\mathbb{R}^2 \setminus ER(v_r,v_s)$ and this gives an standard model (See Figure \ref{MK1pic}) for the configuration of the clasps between $L_r$ and $L_s$ with respect to the other clasps involving $L_r$ or $L_s$. We use this to control the result of sliding $L_r$~over~$L_s$.\\

Without loss of generality assume that $r=1$ and $s=2$. Figure \ref{MK1pic} shows this sliding operation. It is clear that after a number of $R1$ moves, $L_1$ will be an unknot around $v_1$ and $v_2$ and there will be clasps associated with each edge between $v_1$ or $v_2$, and any $v_j$ for $j \neq 1,2$. This means that, after sliding $L_1$ over $L_2$ and deleting $L_2$, we will have a chainmail link on the plane graph $D/\{v_1,v_2\}$, which is the plane graph coming from contracting $v_1$ and $v_2$ to one vertex which we will denote by $v_{1,2}$. This decreases the size of the vertex set by one. Using induction, we know that there exists a spanning tree $T'$ such that $\MK(L_{D/\{v_1,v_2\}},T')$ is an unknot. Spanning tree $T$ can be constructed by replacing $v_{1,2}$ with $v_1$ and $v_2$ and the edge between them, directing the edge from $v_1$ to $v_2$ and putting it as the new maximal edge in the total ordering. 
\end{proof}
\begin{rema}\label{MK1}
    Sliding a component $K$ of the characteristic sublink $C$ over another component $K'$ can be seen as a change of basis of $H_2$. As explained in \cite{GompfStip}, the characteristic sublink which represents the same $Spin$ structure in this new basis is just $C-\{K'\}$. Based on this, one can turn the characteristic sublink into a knot with handle slides in the first step of Kaplan's algorithm. The same reasoning shows that \MK can replace the first step of Kaplan's algorithm when $C$ is connected. When $C$ isn't connected, one can apply this to all connected components of $C$ and end up with an unlink. Then one can slide these unknots over each other and delete one component at each step. This procedure can be done on an arbitrary planar tree with the unknots as its vertices similarly to \MK. 
\end{rema}
\begin{rema}\label{linkingchange}
Note that the linking between the components of the link changes under the handle slides. We can update the linking matrix at each step with the following rule: 
$$\lk(L_{1,2} \ , \ L_i) = \lk(L_1,L_i)+\lk(L_2,L_i).$$
\end{rema}
\begin{proof}[Proof of Theorem \ref{capbound}]
   Assume that $(X,\mathfrak{s})$ is a simply connected negative definite spin filling of $Y=(\Sigma(S^3,K),\mathfrak{t})$.  Let $L$ be the Kirby diagram of $\Sigma(D^4,F_W)$ described in Section 2. Let $L'$ be the characteristic sublink of $L$ associated with $t$. We are going to apply the modified version of Kaplan's algorithm on $L$ using $L'$ which will result in a simply connected spin filling $(X',\mathfrak{s}')$ of $(\Sigma(S^3,K),\mathfrak{t})$. We can take $-X'$ and build the closed 4-manifold $W=X \cup_{Y} (-X')$ which is also spin since $\mathfrak{s}|_{Y}=\mathfrak{s}'|_{Y}=\mathfrak{t}$. Using Furuta's inequality (Theorem \ref{Furuta}) on $W$ gives
   \begin{equation}\label{Furutaused}
      b_2(X)+b_2(-X')\geq \frac{10}{8}|\sigma(X)+\sigma(-X')|+2.
   \end{equation}
   The right-hand side comes from Novikov additivity. Now we need to compute $b_2(-X')$ and $\sigma(-X')$. \\

Based on Lemma \ref{Kaplansimp} and Remark \ref{MK1}, we know that Steps $2,3,4$ of Kaplan's algorithm won't be needed in our setting, and we can easily compute the change of $b_2$ and $\sigma$. First note that in each of the described handle slides, the framings change in the following way. If we assume that framing of $L_p , L_s$ are $r_p, r_s$, respectively, then the framing of the $L_p$ after the sliding will be $r_p + r_s + 2\lk(L_p,L_r)$. Using induction and Remark \ref{linkingchange}, we can see that framing on the final component of the characteristic sublink (after finishing Step $1$) is equal to 
$$\Sigma_{v_i \in C} \ g_{ii} + 2 \Sigma_{\underset{i \neq j}{v_i , v_j \in C}} \ \lk(L_i,L_j) = \Sigma_{v_i \in C} \ g_{ii} + 2 \Sigma_{\underset{i \neq j}{v_i , v_j \in C}} \ g_{ij}.$$ 

Since $g_{ii}=-deg_{W}(v_i)$ and $g_{ij}=|E(v_i,v_j)|$, we will have that $$\Sigma_{v_i \in C} \ g_{ii} + 2 \Sigma_{\underset{i \neq j}{v_i , v_j \in C}} \ g_{ij} = -\Sigma_{v_i \in C} \ deg_{W}(v_i) + 2 \Sigma_{\underset{i \neq j}{v_i , v_j \in C}} \ |E(v_i,v_j)| $$
$$= -|E_{W}(V_C, V_W \setminus V_C)| - 2|E_C| + 2|E_C| = -|E_{W}(V_C, V_W \setminus V_C)|.$$
 
This is the right-hand side of the inequality stated in Theorem \ref{capbound}. Let's denote this value by $-f$. Note that in Step 1, we only use handle slides and isotopies which means that the filling won't change. To calculate the change in $b_2$ and signature, we only need to look at Step $5$. In this step, we blow up $f-1$ meridians in order to turn the characteristic sublink into an unknot with framing $-1$ and then blow down this unknot. These increase $b_2$ and $\sigma$ by $f-2$ and $f$ respectively. Now we only need to use this information in Furuta's inequality. Note that since $X$ is negative definite $\sigma(X)=-b_2(X)$. Assuming that $f \geq 9m$, which is the assumption of Theorem \ref{capbound} (where $m=|V_W|-1$), we can rewrite Equation \ref{Furutaused}, 
$$b_2(X)+m+f-2 \geq \frac{10}{8}|-b_2(X)+m-f|+2$$
$$\Leftrightarrow b_2(X)+m+f-2 \geq \frac{10}{8}(b_2(X)+f-m)+2$$
$$\Leftrightarrow \frac{18}{8}m \geq \frac{2}{8}f+\frac{2}{8}b_2(X)+4.$$
The final inequality is a clear contradiction to $f \geq 9m$. 
\end{proof}
\begin{rema}\label{Ue}
    This procedure gives a generalization of a construction given in \cite{Ue} through Neumann--Seibenmann invariant. Recall that, for a plumbed 3-manifold $Y$, the Neumann--Seibenmann invariant $\bar{\mu}$ is defined as follows. Assume that $\Gamma$ is the plumbing tree and $P(\Gamma)$ is the 4-manifold constructed from plumbing sphere bundles based on $\Gamma$. We know that $Y = \partial P(\Gamma)$. Let $w_s$ be the indicator vector of the characteristic sublink associated with a spin structure $\mathfrak{s}$ on $Y$. Then
    $$\bar{\mu}(Y,\mathfrak{s}) : = \frac{1}{8}(\sigma (P(\Gamma))-\langle w_s, w_s \rangle),$$
     where $\langle  \cdot , \cdot \rangle$ represents the intersection pairing. Ue proves that a Seifert homology sphere $Y$ with Spin structure $\mathfrak{s}$ bounding a negative-definite 4-manifold $X$ with Spin structure $\mathfrak{s}_X$ must satisfy 
    $$\frac{-8\bar{\mu}(Y,\mathfrak{s})}{9} \leq b_2(X) \leq -8 \bar{\mu}(Y,\mathfrak{s}).$$
    Now an obstruction to the existence of simply connected negative-definite spin fillings is $\bar{\mu}(Y,\mathfrak{s}) \geq 0$. In cases when $\widetilde{W}$ is a star-shaped tree (which is the plumbing tree of a Seifert homology $S^3$) one can apply this obstruction to our problem. Whenever $\widetilde{W}$ is a tree, any characteristic subgraph $C$ will be a disjoint union of isolated vertices. As a result, if $w_C \in \mathbb{Z}^m$ is the indicator vector of $V_C$, then 
    \begin{equation}\label{Uecompute}
           \langle w_C , w_C \rangle = \sum_{v_i \in C} - g_{ii} = \sum_{v_i \in C} - deg_W(v_i) = -|E_W(V_C,V_W \setminus V_C)|.
    \end{equation}
    This comes from the fact that $E_C = \emptyset$. Combining Equation \ref{Uecompute} and definition of $\bar{\mu} (Y,\mathfrak{s})$ gives us 
    $$ 8\bar{\mu} (Y,\mathfrak{s}) = -m + |E_W(V_C,V_W \setminus V_C)|.$$
    This means that the obstruction of \cite{Ue} is equivalent to Theorem \ref{capbound}.\\
\end{rema}
This simplification of Kaplan's algorithm and the fact that one can build a characteristic unknot without blowing up or down is of independent interest. The following corollary easily follows from this observation. \\
\begin{coro}\label{characteristicS2}
If a closed 4-manifold $X$ has a chainmail Kirby diagram, then it has a characteristic sphere. 
\end{coro}
This can act as an obstruction for a manifold to have a chainmail diagram. The following is an explicit example of this. This example was pointed out to us by Marco Golla.\\
\begin{exem}\label{charsphere}
    The 4-manifold $X=$ exotic $\mathbb{C}P^2 \# 2 \overline{\mathbb{C}P^2}$ \cite{APexotic} which is symplectic and minimal doesn't have any characteristic spheres. Any embedded sphere $S \subset X$ satisfies $S \cdot S \leq -2$. Adjunction inequality gives us that $S \cdot S \leq 0$, and due to the minimality of $X$ there is no embedded sphere with self-intersection $-1$. Let $S \cdot S = -m$. We use $P_m$ to denote the 4-manifold constructed by a negative linear plumbing of $m-1$ disk bundles with Euler number $2$ over $S^2$. There is a natural diffeomorphism between $\partial P_m$ and $\partial N(S)$. As a result, one can form the closed 4-manifold $M = X \setminus N(S) \bigcup P_m$. 4-manifold $M$ is spin since $S$ is characteristic and $P_m$ is spin. Using Novikov additivity we can deduce that $\sigma(M)=m-1$. Furthermore, $b_2^{-}(M)=1$ since the negative definite part of $H_2(M)$ lies inside $X \setminus N(S)$. This is in contradiction with the classification of intersection forms of indefinite closed spin 4-manifolds. 
\end{exem}
\section{Spin negative definite plumbed fillings}\label{Sec5}
The previous results might lead one to ask if there are any non-special alternating knots $K$ such that $\Sigma(S^3,K)$ has a simply-connected, spin and negative definite filling. A result such as the following theorem might further support this. 
\begin{theo}
    A non-special alternating link $K$ doesn't have a spanning filling which is spin and negative definite. 
\end{theo}
\begin{proof}
    Assume $\Sigma(D^4,F^{+})$ is a negative-definite spin filling. We know that the intersection form of $\Sigma(D^4,F^{+})$ is the Goeritz form of the surface, which means that $G_F$ is a negative definite spanning surface. Using the main theorem of \cite{AltGreene}, we can conclude that $F$ must be the white Tait surface in a diagram of $K$. Then we know, that for $G_F$ to be even, the knot needs to be special, which contradicts the assumption.
\end{proof}
For general fillings, this is far from the truth. We present an example of a non-special alternating knot $K$ with a spin negative definite filling of $\Sigma(S^3,K)$. This example comes from the lens space realization problem and one can generate a family of examples in the same way. We must mention that part of the inspiration for the example comes from \cite{Aceto} which addresses negative definite fillings of lens spaces with minimal $b_2$, but we need to use the lens space fillings that emerge as trace of a surgery on a knot instead of rational homotopy ball fillings. 
\begin{exem}\label{counterexam}
The knot $K$ will be the alternating knot in Figure \ref{counterexampic}.
\begin{figure}[h]
\centering
\begin{center}
\includegraphics[scale=0.3]{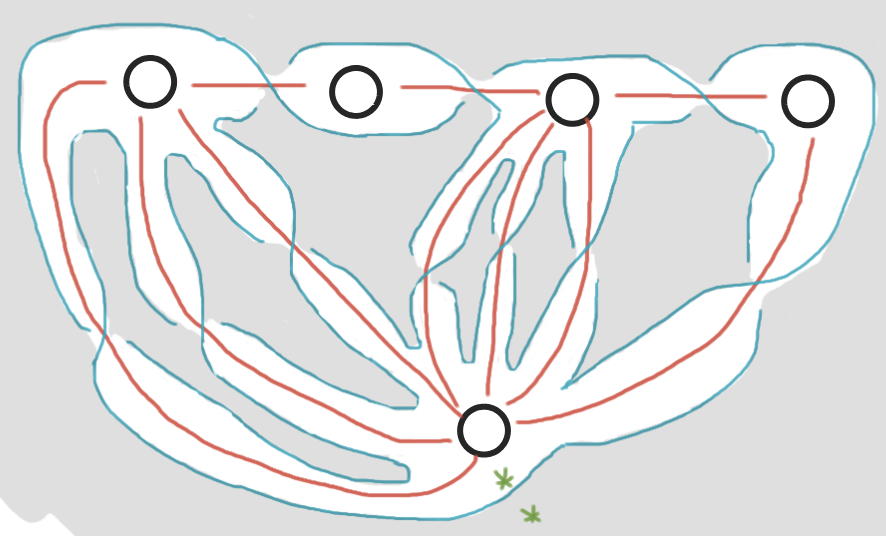}
\end{center}
\caption{Alternating knot K of Example \ref{counterexam}}\label{counterexampic}
\end{figure}\\
The white Tait graph is also drawn in the figure. The reduced white Tait graph will be a single path with framings $(-4,-2,-5,-2)$. Plumbing along this path with these framings will give us the standard negative definite filling of $\Sigma(S^3,K)$. We use the notation $P(a_1, \dots ,a_n)$ to denote the linear plumbing with framings $a_1, \dots , a_n$. For example,
$$\partial P(-4,-2,-5,-2) = \Sigma(S^3,K).$$
We have a standard embedding of $P(-2,-5,-2)$ in $P(-4,-2,-5,-2)$ which comes from the embedding of the plumbing graph of the first 4-manifold in the second. In other words, $P(-4,-2,-5,-2)$ can be constructed from $P(-2,-5,-2)$ by attaching a 2-handle with framing $-4$ to its boundary. The integers $(-2,-5,-2)$ also arises in the following fractional expansion:
$$\frac{16}{9}=2-\frac{1}{5-\frac{1}{2}}.$$
This means that $P(-2,-5,-2)$ bounds $L(16,9)$ (see \cite{Aceto}). We claim that $L(16,9)$ also has a filling in the form of the trace of a knot surgery i.e. 
$$L(16,9) = \partial Tr(K^{-16}).$$
This comes from the description of the Berge knots of type $I_{pm}$. Based on \cite{Lensspace}, picking $i,k \in \mathbb{Z}$ such that $gcd(i,k)=1$ and setting $p=ik\pm1$ and $q \equiv -k^2 \ (mod \ p)$ leads to the lens space $L(p,q)$ which can be realized by a positive surgery on a knot. Setting $i=3$ and $k=5$, gives us $p=16$ and $q \equiv -25 \ (mod \ 16)$, which means that there exists a knot $K$ such that $L(16,7) = K^{+16}$, which in turn means that $L(16,7) = \partial Tr(K^{+16})$. By reversing the orientation, we have $L(16,9)=-L(16,7)= \partial Tr(K^{-16})$.\\

Let $i: P(-2,-5,-2) \hookrightarrow P(-4,-2,-5,-2)=X$ be the aforementioned embedding. Now we construct a 4-manifold $X'$ by deleting the interior of $Im(i)$ from $X$ and gluing $Tr(K^{-16})$ in its place. Then $X'$ will be the result of attaching a $-4$-framed 2-handle to $Tr(K^{-16})$, which means it is simply connected and has $b_2=2$, as it has a handle decomposition with two 2-handles. The intersection form is even as the framing of both 2-handles is even, and hence, $X'$ is spin. Using Novikov additivity we can also prove that $X'$ is negative definite. This is again due to the fact that, while constructing $X'$, we delete a submanifold of $X$ with signature $-3$ and replace it with one with signature $-1$. This finally gives us the example we need. 
\end{exem}
Although we don't have a general result about non-existence of a simply connected negative definite spin filling, but by adding some combinatorial conditions to the Kirby diagram of the filling we can prove such results. The first result of this type is Theorem \ref{plumbed}. This result is directly rooted in Neumann's plumbing calculus. In the following theorem we recall the facts we need from \cite{Neumann}. Note that the branched double covers of links with non-zero determinant are rational homology spheres. As a result, they can be realized as plumbings of disk bundles over surfaces when the base surfaces are all spheres and the plumbing graph is a tree. We can describe these plumbings with a tree with integer weights on vertices and $\pm$ signs on edges. This means that we don't need Neumann's plumbing calculus in its full generality. In the following theorem, we only recall the facts we need.\\

In the rest of the paper, we use the term chain to refer to subgraphs which look like one of the shapes in Figure \ref{chain}.
\begin{figure}[h]
\centering
\begin{center}
\includegraphics[scale=0.4]{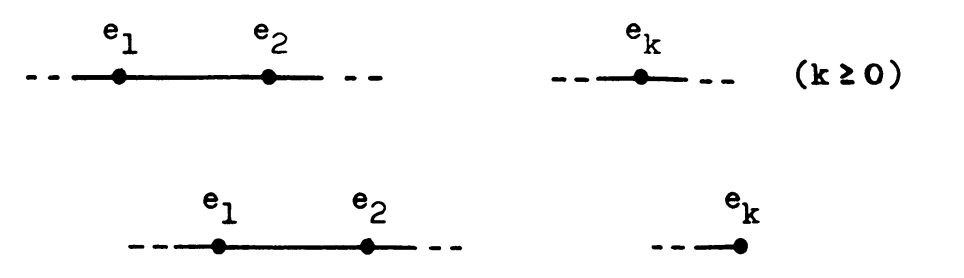}
\end{center}
\caption{A chain in a plumbing graph \cite{Neumann}}\label{chain}
\end{figure}

\begin{theo}[\cite{Neumann}]\label{Neumannclac}
    Any plumbing tree $T$ can be reduced to a normal form using the following moves while keeping the boundary unchanged. \\
    R0. Reverse the sign of all the edges adjacent to a vertex $v$.\\
    R1a. Delete a component consisting of an isolated vertex with weight $\pm 1$.\\ 
    R1b-R3. The moves which are described in Figure \ref{Neumannmove}.\\
    
The normal form is defined by the following properties:\\
N1. None of the operations can be applied, except that $T$ might contain a component like the Figure \ref{N1} with $k \geq 1$ and $e_i \leq -2$.\\
N2. The weights $e_i$ on all chains of $T$ satisfy $e_i \leq -2$.\\
N3. No portion of $T$ has the form shown in Figure \ref{N3}, unless it's in a component of the form shown in Figure \ref{N3'} with $k \geq 1$ and $e_i \leq -2$
\begin{figure}[H]
\centering
\begin{center}
\includegraphics[scale=0.4]{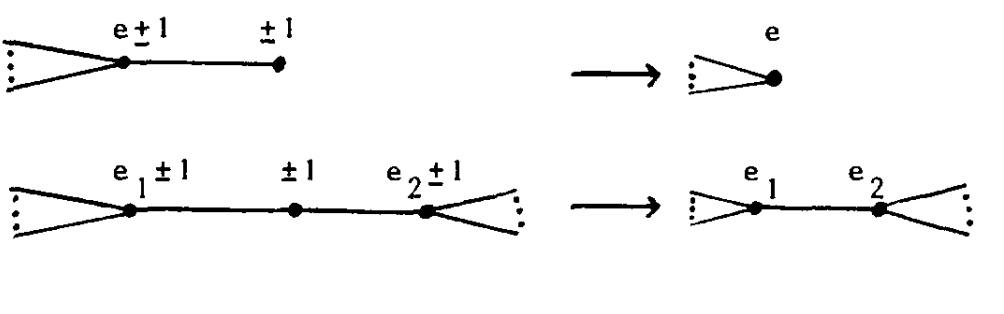}
\end{center}
\begin{center}
\includegraphics[scale=0.4]{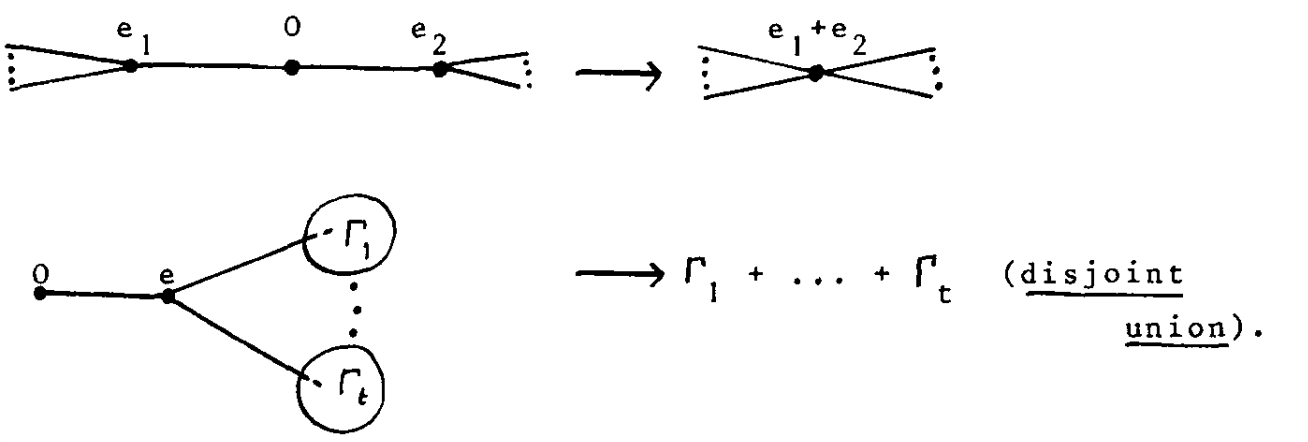}
\end{center}
\caption{Neumann moves \cite{Neumann}}\label{Neumannmove}
\end{figure}
\begin{figure}[H]
\centering
\begin{center}
\includegraphics[scale=0.3]{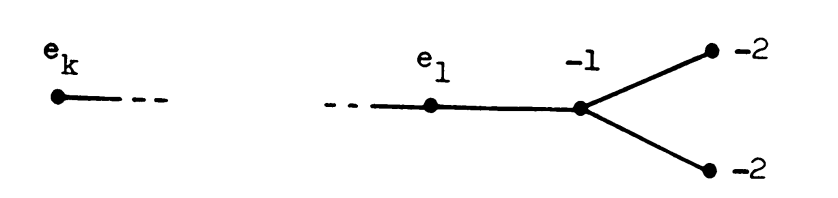}
\end{center}
\caption{Properties of Neumann normal form \cite{Neumann}}\label{N1}
\end{figure}
\begin{figure}[H]
\centering
\begin{center}
\includegraphics[scale=0.3]{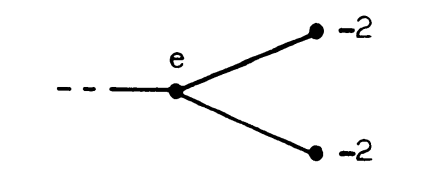}
\end{center}
\caption{Properties of Neumann normal form \cite{Neumann}}\label{N3}
\end{figure}
\begin{figure}[H]
\centering
\begin{center}
\includegraphics[scale=0.3]{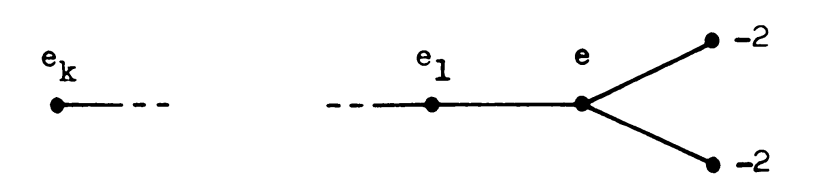}
\end{center}
\caption{Properties of Neumann normal form \cite{Neumann}}\label{N3'}
\end{figure}
\end{theo}
You might notice that the moves described in Theorem \ref{Neumannclac} don't describe the change in the edge signs. When we are dealing with trees, $R0$ gives us that the edge signs don't matter.\\

Before we proceed with proving Theorem \ref{plumbed}, we need to define the excessive property. This part is based on \cite{murasugi}. Recall that a link is called algebraic (in the sense of Conway) if it can be constructed as the boundary of a plumbing of twisted bands. For a weighted tree $T$, we denote the algebraic link constructed from a plumbing based on $T$ by $l(T)$. The tree $T$ is called \emph{negative excessive} if 
$$\forall a_i \in V_T: \ \ \ w(a_i) \leq min \{-2 , -deg_T(a_i)\},$$
The following lemma is proved by Murasugi. 
\begin{lemm}[\cite{murasugi} Proposition 3.3. and 4.1.]\label{excessive}
    For a negative excessive tree $T$, The link $l(T)$ is alternating. Furthermore there is an alternating diagram of $l(T)$ such that $T$ is isomorphic to the reduced white Tait graph $\widetilde{W}$. This isomorphism takes the weights of $T$ to $g_{ii}$ (diagonal of the Goeritz form).
\end{lemm}
With this information, we can proceed with proving Theorem \ref{plumbed}.
\begin{proof}[Proof of Theorem \ref{plumbed}]
We start by proving that any simply connected negative-definite spin plumbed filling is automatically in normal form. Due to the spin condition, we won't have framing $\pm 1$ on any vertex as all framings are even. Due to the negative definite condition, we can't have any vertex with framing $0$ as all framings are negative. This means that conditions $N1$ and $N2$ are satisfied. To show that $N3$ is satisfied, we have to use the odd determinant condition. We know that the determinant of the knot is $|H_1(\Sigma(S^3,K),\mathbb{Z})|$. When the determinant is odd, the 2-torsion vanishes and, as a result, $H^1(\Sigma(S^3,K),\mathbb{Z}_2)=0$, which means that $\Sigma(S^3,K)$ has a unique spin structure. We now use Theorem \ref{Kaplan} to deduce that the number of characteristic sublinks of the Kirby diagram is equal to $1$. This in turn means that the plumbing tree $T$ has a unique characteristic subgraph.\\

We are going to show that in any tree containing the forbidden subgraph of condition $N3$, there exist an even number of characteristic sublinks. A characteristic subgraph $G$ can't contain the parent vertex of the $-2$-framed leaves since the number of edges between a $-2$-framed leaf and $G$ must be even (due to the definition of characteristic sublink). Let us use the names $L=\{l_1,l_2\}$ and $p$ to denote the $-2$ framed leaves and their parent vertex, and  $A:=G \cap \{l_1,l_2\}$. Let $G' =(G - A) \cup (L-A)$. This subgraph is also characteristic. The only change happens with taking the complement of $G \cap L$ on $L$, which means that $E(v,G)$ and $E(v,G')$ are only different for $v\in \{p,l_1,l_2\}$. In all three cases, the parity of $|E(v,G)|$ and $|E(v,G')|$ are the same as $|E(p,G')|=|E(p,G)|-|A|+(2 -|A|)$ and $|E(l_i,G)|=|E(l_i,G')| \pm 2$. This construction builds a bijection on the set of characteristic subgraphs, which means that the size of this set must be even. This gives us condition~$N3$.\\

Let $\widetilde{W}$ be the reduced white Tait graph of $K$. By Lemma \ref{excessive}, We know that $\widetilde{W}$ is isomorphic to the plumbing tree associated to $K$. Using the Tait surgery diagram, we can see that $\Sigma(D^4,F_W)$ is a plumbed filling. We are going to prove that this plumbed filling is also in normal form. The excessive condition forces all weights to be $\leq -2$ and as a result $N1$ and $N2$ are satisfied. Condition $N3$ is satisfied due to the same argument about the parity of the determinant. \\

Now using the uniqueness of Neumann normal form, one can deduce that if a simply connected negative definite spin plumbed filling exists, then its plumbing tree is exactly the reduced white Tait graph. This means that the framings in the white Tait graph; i.e., $g_{ii}$, must be all even, which is equivalent to the knot being special. 
\end{proof} 
The main idea behind Theorem \ref{plumbed} can be generalized to some other types of fillings. 
\begin{defi}
    We call a filling $X$ of a 3-manifold $Y$ a chainmail filling if and only if there exist a Kirby diagram of $X$ which is a chainmail link
\end{defi}

Based on the discussion of Section \ref{Sec2}, the 4-manifold $\Sigma (D^4,F_W)$ always gives a chainmail filling of the branched double cover. Unfortunately, there are no known normal forms for chainmail Kirby diagrams in the literature so the proof of Theorem \ref{plumbed} can't be replicated, but we can use the trick described here which is inspired by \cite{murasugi}.
\begin{defi}
    A weighted planar graph is called \emph{accessible} if it can be realized as the white Tait graph of an alternating link $K$ such that the weights are equal to diagonal entries of the Goeritz matrix of $K$. We call a chainmail filling accessible if it has a chainmail Kirby diagram which is based on an accessible planar graph. 
\end{defi}
The main examples of accessible planar graphs come from the following example: 
\begin{exem}
    Let $D$ be a 2-connected planar graph such that all vertices are adjacent to the unbounded region. Equivalently, any two different cycles of $D$ have at most one common vertex. Furthermore, assume that $D$ is negative excessive; i.e., weights satisfy the following inequality: 
    $$\forall v_i \in V_D: \ \ \ w(v_i) \leq min \{-2 , -deg_D(v_i)\}  .$$
    In this setting, one can add a vertex $\hat{v}$ in the unbounded region and connect it to all $v_i \in V_D$ such that 
    $$|E(\hat{v},v_i)| + deg_D(v_i) = |w(v_i)|.$$
    The median construction on $D \cup \{\hat{v}\}$ gives an alternating link such that the reduced white Tait graph is isomorphic to $D$ and the weights of $D$ will become the diagonal entries of Goeritz matrix. 
\end{exem}

\begin{theo}
    Let $K$ be an alternating link. Then $\Sigma(S^3,K)$ admits a simply connected negative definite spin accessible filling if and only if $K$ is special alternating
\end{theo}
\begin{proof}
    Assume such a filling $X$ exists and it has a chainmail diagram based on an accessible plane graph like $D$. Let $K'$ be an alternating link with $\widetilde{W}_{K'} = D$. This means that the chainmail Kirby diagram based on $D$ is also a Kirby diagram for $\Sigma(S^3,K')$, which means that branched double covers of $K$ and $K'$ are diffeomorphic. By a result of Greene \cite{greenemutation}, we can deduce that $K$ and $K'$ are mutants. Planar mutation of alternating knots preserves the number of white regions of the diagram and as a result $$b_2(X) = |V_D| = |V_{\widetilde{W}_{K'}}|=|V_{\widetilde{W}_K}|.$$
   Based on Theorem \ref{specialchar}, we can deduce that $K$ is special alternating.  
\end{proof}

\bibliography{bibtemplate}
\bibliographystyle{smfalpha} 

\end{document}